\documentclass{amsart}

\usepackage{amsmath,amssymb,amsthm,mathtools}
\usepackage{microtype}
\usepackage{enumitem}
\usepackage[hidelinks]{hyperref}

\newtheorem{theorem}{Theorem}[section]
\newtheorem{proposition}[theorem]{Proposition}
\newtheorem{lemma}[theorem]{Lemma}
\newtheorem{corollary}[theorem]{Corollary}
\theoremstyle{definition}
\newtheorem{definition}[theorem]{Definition}
\theoremstyle{remark}

\newcommand{\Fp}{\mathbb F_p}
\newcommand{\Fq}{\mathbb F_q}
\newcommand{\op}{\mathrm{op}}

\title[Quadratic Expansion over Prime Fields]
{Quadratic Expansion over Prime Fields via Centered Collisions and Popular-Sum Amplification}

\author{Yao Zhi}
\email{ericyao2026@gmail.com}


\begin{document}

\begin{abstract}
Let $p$ be an odd prime, let $\varnothing\neq A\subseteq\Fp$ have cardinality
$N$, and let $f\in\Fp[x,y]$ be a non-degenerate quadratic polynomial.
Writing $S=|A+A|$ and $M=|f(A,A)|$, we prove the full-range trade-off
\[
 S^8M^6\gtrsim
 \frac{N^{17}}{(1+N^3/p^2)^3}.
\]
Consequently,
\[
 \max\{|A+A|,|f(A,A)|\}
 \gtrsim
 \min\{N^{17/14},\,p^{3/7}N^{4/7}\},
\]
and in particular the exponent $17/14$ holds throughout $N\le p^{2/3}$.
The proof combines a centered collision estimate for
$F(u,v,w)=f(u+v,w)$, a mixed fourth-energy bound, and a popular-sum
amplification.  Two complementary incidence estimates enter the
argument: a centered spectral bound in the dense collision regime and
a point--plane bound in the sparse regime.
\end{abstract}

\maketitle

\section{Introduction and main results}

The sum--product principle, originating in work of Erd\H{o}s and
Szemer\'edi \cite{ErdosSzemeredi} and developed over finite fields by
Bourgain, Katz and Tao \cite{BKT}, asserts that a set with strong
additive structure must expand under a genuinely nonlinear operation.  Quantitative finite-field developments include
\cite{Garaev,Vu,RRSh}; Fourier-analytic and polynomial expansion
methods appear in \cite{HartLiShen,BukhTsimerman,Tao}, while recent
work continues to connect spectral incidence geometry with polynomial
expansion \cite{AralaChow}.

For two-variable conditional expanders, one studies lower bounds for
$\max\{|A+A|,|f(A,A)|\}$; see, for example,
\cite{HegyvaryHennecart,MojarradPham}.  For every non-degenerate
quadratic polynomial, Koh, Nassajian Mojarrad, Pham and Valculescu
\cite[Theorem~1.5]{KMPV} proved the exponent $6/5$ for
$|A|\le p^{5/8}$.  Related three-variable quadratic expansion and
finite-field incidence methods were developed by Pham, Vinh and de
Zeeuw \cite{PhamVinhDeZeeuw}.  Using the higher-energy method of Shakan and Shkredov
\cite{ShakanShkredov}, Mirzaei \cite{Mirzaei} obtained the exponent
$74/61$ for $|A|\le p^{1/2}$ through the mixed fourth-energy quantity
$d_4^+(A)$.  Mohammadi and Stevens later obtained the stronger
exponent $28/23$ for $\max\{|A-A|,|f(A,A)|\}$ in the smaller range
$|A|\ll p^{23/52}$ \cite[Theorem~4]{MohammadiStevensLow}; this does not
imply an $A+A$ estimate.  Related two-variable expander results include \cite{TranNguyen}.
Our argument combines a centered spectral incidence estimate, Rudnev's
sparse point--plane theorem \cite{Rudnev}, and the additive
double-counting inequality recorded in
\cite[Proposition~1]{MohammadiStevens}, whose proof is extracted there
from Rudnev, Shakan and Shkredov \cite{RudnevShakanShkredov}.

We use the following standard notion of degeneracy.

\begin{definition}\label{def:nondegenerate}
Let $\mathbb F$ be a field.  A polynomial $f\in\mathbb F[x,y]$ is \emph{non-degenerate} if there do not
exist $h\in\mathbb F[t]$, $(\lambda,\mu)\neq(0,0)$, and $\tau\in\mathbb F$ such that
$f(x,y)=h(\lambda x+\mu y+\tau)$.
\end{definition}

Throughout the prime-field part of the paper, write
$N=|A|$, $S=|A+A|$, and $M=|f(A,A)|$.  We use $X\ll Y$ and $X\gg Y$
for absolute-constant inequalities.  The notation $X\lesssim Y$ and
$X\gtrsim Y$ may additionally suppress powers of logarithms coming
from dyadic decompositions, and $X\approx Y$ means that both
$X\lesssim Y$ and $Y\lesssim X$ hold.

The main theorem is most naturally stated as a product trade-off.

\begin{theorem}[Full-range quadratic expansion]\label{thm:main}
Let $p$ be an odd prime, let $\varnothing\neq A\subseteq\Fp$, and let
$f\in\Fp[x,y]$ be a non-degenerate quadratic polynomial.  Then
\begin{equation}\label{eq:main-product}
 S^8M^6\gtrsim
 \frac{N^{17}}{(1+N^3/p^2)^3}.
\end{equation}
Consequently,
\begin{equation}\label{eq:main-max}
 \max\{S,M\}
 \gtrsim
 N^{17/14}(1+N^3/p^2)^{-3/14}
 \gtrsim
 \min\{N^{17/14},\,p^{3/7}N^{4/7}\}.
\end{equation}
\end{theorem}

The sparse-range statement of primary interest follows immediately.

\begin{corollary}[The $17/14$ range]\label{cor:seventeen-fourteen}
Under the hypotheses of Theorem~\ref{thm:main}, if $N\le p^{2/3}$, then
\[
 \max\{|A+A|,|f(A,A)|\}\gtrsim N^{17/14}.
\]
\end{corollary}

Corollary~\ref{cor:seventeen-fourteen} improves the previous general
quadratic $A+A$ exponent $74/61$ from \cite{Mirzaei} and extends the
range from $N\le p^{1/2}$ to $N\le p^{2/3}$.  The latter is exactly the
scale $N^3\le p^2$ associated with the Cartesian point--plane incidence
threshold.  For $N=p^\theta$ with $\theta>2/3$,
Theorem~\ref{thm:main} gives the exponent $4/7+3/(7\theta)$ relative
to $N$, still strictly larger than $6/5$ when $\theta<15/22$.
To the best of our knowledge, Corollary~\ref{cor:seventeen-fourteen}
is the first estimate for all non-degenerate quadratic polynomials that
is strictly stronger than $6/5$ throughout the range $N\le p^{2/3}$;
compare \cite{KMPV,Mirzaei,MohammadiStevensLow}.
For $N>p^{2/3}$, the second term in \eqref{eq:main-max} does not
exceed $p$, as required by the ambient field size.

The product estimate also yields the following one-sided consequence.

\begin{corollary}[Small additive doubling]\label{cor:small-doubling}
Under the hypotheses of Theorem~\ref{thm:main}, if $|A+A|\le KN$, then
\[
 |f(A,A)|
 \gtrsim
 K^{-4/3}N^{3/2}(1+N^3/p^2)^{-1/2}
 \gtrsim
 K^{-4/3}\min\{N^{3/2},p\}.
\]
\end{corollary}

The proof has three ingredients.  First, collisions of
$F(u,v,w)=f(u+v,w)$ are converted into weighted point--plane incidences.
The centered incidence input is a weighted normalized form of Vinh's
spectral point--hyperplane estimate \cite{Vinh}.  We record a short
Fourier proof because the exact uniform term is essential and the
estimate has no $p^2$ support restriction.  Second, below that threshold we use
Rudnev's point--plane theorem \cite{Rudnev}; see \cite{RRSh} for a
representative sum--product application and \cite{StevensDeZeeuw} for
complementary planar incidence theory.  These two estimates give
Theorem~\ref{thm:fourth-energy}, which is finally inserted into
\cite[Proposition~1]{MohammadiStevens}.

\paragraph{\textbf{The role of AI in this work}}

The authors used ChatGPT 5.5 Plus as an auxiliary research tool during the development of this work. In particular, the formulation and proof strategy of Lemma 2.1 were substantially inspired by suggestions generated through interactions with this AI tool. ChatGPT 5.5 Plus was also used to assist with some algebraic computations, consistency checks, and intermediate calculations arising in the course of the proof. All mathematical statements, arguments, and computations used in the final manuscript were subsequently checked and verified by the authors, who take full responsibility for the correctness of the results and for the final content of the paper.

\section{Centered quadratic collisions}

We first work over an arbitrary finite field $\Fq$ of odd
characteristic.  This section is independent of the prime-field
incidence theorem used later.

\subsection{A weighted Vinh-type point--plane estimate}

Let $\mathcal P=\Fq^3$ and let $\mathcal H$ be the family of normalized
affine planes
\[
 H_{\alpha,\beta,\gamma}
 =
 \{(x,y,z)\in\Fq^3:\alpha x+\beta y+z=\gamma\},
 \qquad
 (\alpha,\beta,\gamma)\in\Fq^3.
\]
For non-negative weights $r:\mathcal P\to\mathbb R$ and
$s:\mathcal H\to\mathbb R$, write
\[
 I(r,s)
 =
 \sum_{P\in\mathcal P}\sum_{\substack{H\in\mathcal H\\P\in H}}
 r(P)s(H).
\]

Vinh's point--hyperplane theorem gives the corresponding centered
unweighted estimate \cite{Vinh}.  The incidence principle is therefore
classical; we record the following weighted normalized form because it
fits the collision parametrization exactly and follows from the same
spectral mechanism.

\begin{lemma}[Weighted Vinh-type incidence estimate]
\label{lem:vinh-weighted}
For the preceding weights,
\begin{equation}\label{eq:vinh-weighted}
 \left|
 I(r,s)-\frac{\|r\|_1\|s\|_1}{q}
 \right|
 \le q\|r\|_2\|s\|_2.
\end{equation}
\end{lemma}

\begin{proof}
Fix a nontrivial additive character $\chi$ of $\Fq$ and let
$T:\ell^2(\mathcal H)\to\ell^2(\mathcal P)$ be the incidence operator
\[
 (Ts)(x,y,z)
 =
 \sum_{\alpha,\beta\in\Fq}
 s(H_{\alpha,\beta,\alpha x+\beta y+z}).
\]
The normalized characters
$\psi_{\xi,\eta,\tau}(\alpha,\beta,\gamma)
=q^{-3/2}\chi(\xi\alpha+\eta\beta+\tau\gamma)$ form an orthonormal
basis of $\ell^2(\mathcal H)$, and
\[
 (T\psi_{\xi,\eta,\tau})(x,y,z)
 =
 q^{-3/2}\chi(\tau z)
 \sum_{\alpha\in\Fq}\chi((\xi+\tau x)\alpha)
 \sum_{\beta\in\Fq}\chi((\eta+\tau y)\beta).
\]
If $\tau=0$, this vanishes unless $\xi=\eta=0$, in which case its
norm is $q^2$.  If $\tau\neq0$, it is supported on
$x=-\xi/\tau$, $y=-\eta/\tau$, has norm $q$, and the nonzero images
are mutually orthogonal.  Thus $T$ has singular value $q^2$ on the
constant subspace and operator norm $q$ on its orthogonal complement.
Since $T1_{\mathcal H}=q^2 1_{\mathcal P}$, the constant components
contribute $\|r\|_1\|s\|_1/q$.  Cauchy--Schwarz on the orthogonal
complements proves \eqref{eq:vinh-weighted}.
\end{proof}

\subsection{The centered collision estimate}

For a quadratic polynomial $h$ and finite $U,V,W\subseteq\Fq$, put
$F_h(u,v,w)=h(u+v,w)$ and let $E_{F_h}(U,V,W)$ denote the number of
pairs of triples in $(U\times V\times W)^2$ having the same
$F_h$-value.  If
\[
 f(x,y)=a_0x^2+b_0y^2+c_0xy+\ell_{1,0}x+\ell_{2,0}y+\kappa_0,
\]
write $f^{\op}(x,y)=f(y,x)$.  Non-degeneracy is invariant under this
transposition, and $f^{\op}(C,C)=f(C,C)$ for every $C\subseteq\Fq$.

\begin{theorem}[Oriented centered quadratic collision estimate]
\label{thm:centered-collision}
Let $f\in\Fq[x,y]$ be a non-degenerate quadratic polynomial and let
$X=|U||V||W|$.

If $(a_0,b_0)\neq(0,0)$, choose $g\in\{f,f^{\op}\}$ whose
$x^2$-coefficient is nonzero.  Then
\[
 \left|
 E_{F_g}(U,V,W)-\frac{X^2}{q}
 \right|
 \le 2qX.
\]

If $a_0=b_0=0$, then non-degeneracy implies $c_0\neq0$.  Put $g=f$,
$w_*=-\ell_{1,0}/c_0$, $W^\circ=W\setminus\{w_*\}$, and
$X^\circ=|U||V||W^\circ|$.  Then
\[
 \left|
 E_{F_g}(U,V,W^\circ)-\frac{(X^\circ)^2}{q}
 \right|
 \le qX^\circ.
\]
\end{theorem}

\begin{proof}
First suppose that the $x^2$-coefficient of $g$ is nonzero.  Write
\[
 g(x,y)=ax^2+by^2+cxy+\ell_1x+\ell_2y+\kappa,
 \qquad a\neq0.
\]
Completing the square in $x$, define
\[
 \phi(w)=\frac{cw+\ell_1}{2a},
 \qquad
 \psi(w)=bw^2+\ell_2w+\kappa-\frac{(cw+\ell_1)^2}{4a}.
\]
Then $g(x,w)=a(x+\phi(w))^2+\psi(w)$.  The polynomial $\psi$ is
nonconstant.  Indeed, if $\psi$ were constant, then $g$ would be a
univariate quadratic polynomial in the affine form $x+\phi(y)$,
contrary to Definition~\ref{def:nondegenerate}.  Hence every fibre of
$\psi$ has cardinality at most two.

For $\sigma(v,w)=v+\phi(w)$, define
\[
 P(u,v',w')
 =
 \bigl(u,\,2a\sigma(v',w'),\,au^2-a\sigma(v',w')^2-\psi(w')\bigr)
\]
and the normalized plane
\[
 \Pi(u',v,w):
 \quad
 2a\sigma(v,w)x_1-u'x_2+x_3
 =
 au'^2-a\sigma(v,w)^2-\psi(w).
\]
Substitution shows that
\[
 P(u,v',w')\in\Pi(u',v,w)
 \quad\Longleftrightarrow\quad
 F_g(u,v,w)=F_g(u',v',w').
\]
The first two coordinates of a parameter point determine $u$ and
$\sigma(v',w')$, while the third then determines $\psi(w')$.  Thus the point
parametrization has multiplicity at most two, and the same argument
applies to the plane parametrization.  If $r$ and $s$ are the induced
weights, then $\|r\|_1=\|s\|_1=X$ and
$\|r\|_2^2,\|s\|_2^2\le2X$.  Lemma~\ref{lem:vinh-weighted} gives the asserted bound.

It remains to treat the pure mixed case
\[
 g(x,y)=cxy+\ell_1x+\ell_2y+\kappa,
 \qquad c\neq0.
\]
For $w\in W^\circ$ one has $cw+\ell_1\neq0$ and
\[
 F_g(u,v,w)=(cw+\ell_1)(u+v)+\ell_2w+\kappa.
\]
Define
\[
 P(w,u',v)
 =
 \bigl(w,\,u',\,(cw+\ell_1)v+\ell_2w-\ell_1u'\bigr)
\]
and
\[
 \Pi(w',u,v'):
 \quad
 cu x_1-cw'x_2+x_3
 =
 (cw'+\ell_1)v'+\ell_2w'-\ell_1u.
\]
Again, direct substitution gives the exact collision--incidence
correspondence.  Both parametrizations are injective: the first two
coordinates determine $w,u'$ on the point side, while the coefficients
of $x_1,x_2$ determine $u,w'$ on the plane side, and
$cw+\ell_1\neq0$ then determines the remaining variable.  Hence the
induced weights are $\{0,1\}$-valued, and
Lemma~\ref{lem:vinh-weighted} gives the asserted bound.
\end{proof}

The deletion of $w_*$ in the pure mixed case is necessary, since on
that slice $F_g(u,v,w_*)$ is independent of $u+v$.

The centered estimate has several immediate consequences.  In the
first case of Theorem~\ref{thm:centered-collision}, put
$\Omega=U\times V\times W$ and $C_0=2$; in the pure mixed case put
$\Omega=U\times V\times W^\circ$ and $C_0=1$.  Write
$X_\Omega=|\Omega|$ and
\[
 \nu_\Omega(z)=|\{\omega\in\Omega:F_g(\omega)=z\}|.
\]

\begin{corollary}[Variance, discrepancy, and image size]
With the preceding notation,
\begin{equation}\label{eq:variance}
 \sum_{z\in\Fq}
 \left(\nu_\Omega(z)-\frac{X_\Omega}{q}\right)^2
 \le C_0qX_\Omega.
\end{equation}
Consequently, for every $Y\subseteq\Fq$,
\begin{equation}\label{eq:discrepancy}
 \left|
 \sum_{z\in Y}\nu_\Omega(z)-\frac{|Y|X_\Omega}{q}
 \right|
 \le
 \sqrt{C_0q|Y|X_\Omega}.
\end{equation}
Moreover,
\[
 |F_g(\Omega)|
 \ge
 \frac{qX_\Omega}{X_\Omega+C_0q^2},
 \qquad
 |\Fq\setminus F_g(\Omega)|
 \le
 \frac{C_0q^3}{X_\Omega},
\]
and for every $\varepsilon>0$ the number of $z\in\Fq$ satisfying
$|\nu_\Omega(z)-X_\Omega/q|\ge\varepsilon X_\Omega/q$ is at most
$C_0q^3/(\varepsilon^2X_\Omega)$.
\end{corollary}

\begin{proof}
Expanding the square in \eqref{eq:variance} gives
$E_{F_g}(\Omega)-X_\Omega^2/q$, so
Theorem~\ref{thm:centered-collision} proves the variance bound.
Cauchy--Schwarz gives \eqref{eq:discrepancy}.  The image lower bound
follows from $X_\Omega^2\le |F_g(\Omega)|E_{F_g}(\Omega)$.
Applying \eqref{eq:discrepancy} to the missing-value set gives the
second estimate, while Chebyshev's inequality applied to
\eqref{eq:variance} gives the final assertion.
\end{proof}

\section{The fourth-energy estimate}

We now return to the prime field $\Fp$.  Let $f_0$ be a non-degenerate
quadratic polynomial and let $g\in\{f_0,f_0^{\op}\}$ be the orientation
fixed in Theorem~\ref{thm:centered-collision}.  Since
$g(A,A)=f_0(A,A)$, write $N=|A|$ and $M=|f_0(A,A)|$.

For finite $U,V\subseteq\Fp$, let
\[
 r_{U-V}(z)=|\{(u,v)\in U\times V:u-v=z\}|
\]
and, for $k>0$,
\[
 E_k^+(U,V)=\sum_z r_{U-V}(z)^k.
\]
We abbreviate $E_k^+(U)=E_k^+(U,U)$ and
$E^+(U,V)=E_2^+(U,V)$.  Following the mixed-energy formulation used
in \cite{Mirzaei}, for nonempty $A$ define
\[
 d_4^+(A)
 =
 \sup_{\varnothing\neq U\subseteq\Fp}
 \frac{E_4^+(A,U)}{|A||U|^3}.
\]

We prove the following full-range estimate.

\begin{theorem}[Centered fourth-energy estimate]
\label{thm:fourth-energy}
For every nonempty $A\subseteq\Fp$,
\begin{equation}\label{eq:d4-full}
 d_4^+(A)
 \lesssim
 \frac{M^2}{N^2}+\frac{M^2N}{p^2}.
\end{equation}
Consequently, for every nonempty $U\subseteq\Fp$,
\begin{equation}\label{eq:mixed-energy-full}
 E_4^+(A,U)
 \lesssim
 \left(\frac{M^2}{N^2}+\frac{M^2N}{p^2}\right)N|U|^3.
\end{equation}
\end{theorem}

We use two elementary or standard inputs.  First, $M\ge N/2$.
Indeed, if the chosen orientation has nonzero $x^2$-coefficient, then
$x\mapsto g(x,y)$ has fibres of size at most two for every fixed $y$.
In the pure mixed case, for $N\ge2$ one may choose
$y\in A\setminus\{w_*\}$, and then $x\mapsto g(x,y)$ is injective.

The second input is Rudnev's point--plane incidence theorem
\cite{Rudnev}; see also \cite{RRSh} for a sum--product application.
We use the following standard consequence of Rudnev's projective
theorem, after affine restriction and projective duality.

\begin{theorem}[Rudnev]\label{thm:rudnev}
There exists an absolute constant $c_0>0$ such that the following
holds.  Let $\mathcal Q$ be a finite point set and $\Pi$ a finite set of planes
in $\Fp^3$, with $|\mathcal Q|\le|\Pi|$ and
$|\mathcal Q|\le c_0p^2$.  If no affine line contains more than $k$
points of $\mathcal Q$, then
\[
 I(\mathcal Q,\Pi)
 \ll
 |\mathcal Q|^{1/2}|\Pi|+k|\Pi|.
\]
The dual statement, obtained by projective duality, also holds.
\end{theorem}

For $F(u,v,w)=g(u+v,w)$ we need the following sparse collision bound.

\begin{lemma}[Sparse collision estimate]
\label{lem:small-collision}
Let $U,V,W\subseteq\Fp$ be nonempty; in the pure mixed case take
$W\subseteq\Fp\setminus\{w_*\}$.  Set $X=|U||V||W|$ and
$K=\max\{|U|,\allowbreak |V|,\allowbreak |W|\}$.  If $X\le c_0p^2$, then
\begin{equation}\label{eq:small-collision}
 E_F(U,V,W)\lesssim X^{3/2}+KX.
\end{equation}
\end{lemma}

\begin{proof}
Suppose first that the oriented polynomial has nonzero $x^2$-coefficient
and use the notation from the proof of
Theorem~\ref{thm:centered-collision}.  Let $\mathcal Q_0$ and $\Pi_0$
be the sets of distinct parameter points and planes.  Their
parametrization multiplicities are at most two, so
$X/2\le|\mathcal Q_0|,|\Pi_0|\le X$, and
$E_F(U,V,W)\le4I(\mathcal Q_0,\Pi_0)$.

The point set is contained in
\[
 \left\{
 \bigl(u,2a(v+\phi(w)),
 au^2-a(v+\phi(w))^2-\psi(w)\bigr):
 u\in U,\ v\in V,\ w\in W
 \right\}.
\]
Every affine line contains $O(K)$ points of this set.  If the first
coordinate is nonconstant on the line, it is injective there.  If it
is constant, say $u=u_0$, then the line is given in that coordinate
plane by $\lambda Y+\rho Z=\omega$.  For each fixed $w$, if $\rho=0$
the equation determines at most one $v$, while if $\rho\neq0$ it is a
quadratic equation in $v$ with nonzero leading coefficient
$-\rho a$.  Identify a normalized plane $\alpha x+\beta y+z=\gamma$ with its
coefficient point $(\alpha,\beta,\gamma)$ in the affine chart of
projective dual space.  The coefficient-point set corresponding to
$\Pi_0$ is sent by the invertible map
$(x,y,z)\mapsto(-y,x,z)$ to a set of the same form.  Thus the dual
family also has line-richness $O(K)$.
Theorem~\ref{thm:rudnev}, after dualizing if necessary, proves
\eqref{eq:small-collision}.

In the pure mixed case both parametrizations in
Theorem~\ref{thm:centered-collision} are injective.  The point family
is contained in
\[
 \{(w,u,(cw+\ell_1)v+\ell_2w-\ell_1u):
 w\in W,\ u\in U,\ v\in V\}.
\]
If either of the first two coordinates varies on a line, that
coordinate is injective; if both are constant, the distinct points are
parametrized injectively by $v$ because $cw+\ell_1\neq0$.  Hence the
line-richness is at most $K$.  The dual
coefficient points are carried to the same form by the invertible map
$(x,y,z)\mapsto(-y/c,x/c,z)$, so Theorem~\ref{thm:rudnev} again gives
\eqref{eq:small-collision}.
\end{proof}

\subsection{Proof of the fourth-energy estimate}

\begin{proof}[Proof of Theorem~\ref{thm:fourth-energy}]
The case $N=1$ is immediate, so assume $N\ge2$.  Fix a nonempty
$B\subseteq\Fp$ and write $L=|B|$.  A dyadic decomposition of
$r_{A-B}$ gives a set $D\subseteq\Fp$, with $d=|D|$, and a dyadic
parameter $t\ge1$ such that
\begin{equation}\label{eq:dyadic-reduction}
 E_4^+(A,B)\lesssim dt^4,
 \qquad
 dt\le NL,
 \qquad
 t\le\min\{N,L\}.
\end{equation}
Put
\begin{equation}\label{eq:Q-def}
 Q_0=\frac{dt^4}{NL^3}.
\end{equation}
It is enough to bound $Q_0$.

In the pure mixed case set $A^\circ=A\setminus\{w_*\}$; otherwise set
$A^\circ=A$.  Write $N^\circ=|A^\circ|$.  Since $N\ge2$,
$N/2\le N^\circ\le N$.  Let $Y=g(A,A)=f_0(A,A)$, $P=dLN^\circ$, and,
for $z\in\Fp$, let
\[
 \mathcal N(z)
 =
 |\{(u,v,w)\in D\times B\times A^\circ:F(u,v,w)=z\}|.
\]
For each $u\in D$ there are at least $t$ pairs $(a,b)\in A\times B$
with $u=a-b$.  For every such pair and every $w\in A^\circ$,
$F(u,b,w)=g(a,w)\in Y$.  Hence
\begin{equation}\label{eq:representation-lower}
 \sum_{z\in Y}\mathcal N(z)\ge dtN^\circ.
\end{equation}

\emph{Small-product case.}
Assume $P\le c_0p^2$.  By Cauchy--Schwarz,
\eqref{eq:representation-lower}, and Lemma~\ref{lem:small-collision},
\[
 (dtN^\circ)^2
 \lesssim
 M(P^{3/2}+KP),
 \qquad
 K=\max\{d,L,N^\circ\}.
\]
If $P^{3/2}$ dominates, then
$d^{1/2}t^2(N^\circ)^{1/2}\lesssim ML^{3/2}$; after squaring,
$dt^4N^\circ\lesssim M^2L^3$.  Since $N^\circ\ge N/2$,
$Q_0\lesssim M^2/N^2$.

Suppose that $KP$ dominates.  Then
$dt^2N^\circ\lesssim MKL$.  If $K=d$, then
$t^2N^\circ\lesssim ML$ and, using $dt\le NL$,
\[
 Q_0\le\frac{t^3}{L^2}
 =
 \frac{t^2N^\circ}{ML}\,
 \frac{Mt}{N^\circ L}
 \lesssim\frac{M}{N}.
\]
If $K=L$, then $dt^2N^\circ\lesssim ML^2$ and
$t^2\le NL\le2N^\circ L$, whence
\[
 Q_0
 =
 \frac{dt^2N^\circ}{ML^2}\,
 \frac{Mt^2}{NN^\circ L}
 \lesssim\frac{M}{N}.
\]
If $K=N^\circ$, then $dt^2\lesssim ML$, so
\[
 Q_0
 =
 \frac{dt^2}{ML}\,
 \frac{Mt^2}{NL^2}
 \lesssim\frac{M}{N},
\]
using $t\le L$.  Since $M\ge N/2$, throughout the small-product case
\[
 Q_0\lesssim\frac{M^2}{N^2}.
\]

\emph{Large-product case.}
Assume $P>c_0p^2$.  Applying
\eqref{eq:discrepancy} with
$\Omega=D\times B\times A^\circ$ and output set $Y$, and using
\eqref{eq:representation-lower}, gives
\begin{equation}\label{eq:large-discrepancy}
 dtN^\circ
 \lesssim
 \frac{MP}{p}+\sqrt{MpP}.
\end{equation}
If the uniform term dominates, then $t\lesssim ML/p$.  From
\eqref{eq:dyadic-reduction} and \eqref{eq:Q-def},
\[
 Q_0\le\frac{t^3}{L^2}
 \lesssim\frac{M^3L}{p^3},
 \qquad
 Q_0\le\frac{N^3}{L^2}.
\]
The second inequality gives
$L\le N^{3/2}Q_0^{-1/2}$.  Substituting this into the first yields
$Q_0^{3/2}\lesssim M^3N^{3/2}/p^3$, and hence
\[
 Q_0\lesssim\frac{M^2N}{p^2}.
\]

Suppose instead that the fluctuation term in
\eqref{eq:large-discrepancy} dominates.  Squaring gives
\begin{equation}\label{eq:fluctuation-one}
 dt^2N^\circ\lesssim MpL.
\end{equation}
By Theorem~\ref{thm:centered-collision},
\[
 E_F(D,B,A^\circ)
 \lesssim
 \frac{P^2}{p}+pP
 \lesssim
 \frac{P^2}{p},
\]
because $P>c_0p^2$ and $c_0$ is an absolute constant.  Cauchy--Schwarz and
\eqref{eq:representation-lower} therefore give
\begin{equation}\label{eq:fluctuation-two}
 pt^2\lesssim ML^2.
\end{equation}
Using \eqref{eq:fluctuation-one} and
\eqref{eq:fluctuation-two} in the exact factorization
\[
 Q_0
 =
 \left(\frac{dt^2N^\circ}{MpL}\right)
 \left(\frac{pt^2}{ML^2}\right)
 \left(\frac{M^2}{NN^\circ}\right)
\]
gives $Q_0\lesssim M^2/N^2$.

The two cases are exhaustive, so
\[
 Q_0
 \lesssim
 \frac{M^2}{N^2}+\frac{M^2N}{p^2}.
\]
Together with \eqref{eq:dyadic-reduction} and
\eqref{eq:Q-def}, this proves \eqref{eq:d4-full}; the definition of
$d_4^+(A)$ then gives \eqref{eq:mixed-energy-full}.
\end{proof}

\section{Popular-sum amplification}

We use one external additive-combinatorial input, namely the part of
\cite[Proposition~1]{MohammadiStevens} needed below.  Mohammadi and
Stevens explicitly extract this purely additive double-counting
inequality from the argument of Rudnev, Shakan and Shkredov
\cite{RudnevShakanShkredov}; it has no finite-field size hypothesis.
The dyadic losses are absorbed by $\lesssim$.

\begin{proposition}[Popular-sum double counting]\label{prop:popular-sum}
Let $A\subseteq\Fp$, $|A|=N$, and $S=|A+A|$.  There exist a set
$B\subseteq A$ with $|B|\gg N$, nonempty sets
$F\subseteq B-B$ and $E\subseteq A-F$, and dyadic parameters
$\nu,\mu\ge1$ such that
\[
 E_{4/3}^+(B)\approx |F|\nu^{4/3},
 \qquad
 \nu\le r_{B-B}(d)<2\nu\quad(d\in F),
\]
\[
 \mu\le r_{A-F}(e)<2\mu\quad(e\in E),
\]
and
\begin{equation}\label{eq:popular-operator}
 E_{4/3}^+(B)^3
 \lesssim
 \frac{
 S^8 E_4^+(A)^2 E_4^+(A,E)\mu^4\nu^4
 }{N^{24}}.
\end{equation}
\end{proposition}

\begin{proof}[Proof of Theorem~\ref{thm:main}]
For bounded $N$, the result follows after adjusting the implicit
constant.  We may therefore assume that the dyadic selections in
Proposition~\ref{prop:popular-sum} are available.  Put
\[
 R=1+\frac{N^3}{p^2}.
\]
By Theorem~\ref{thm:fourth-energy},
\[
 E_4^+(A)\lesssim M^2N^2R,
 \qquad
 E_4^+(A,E)\lesssim\frac{M^2|E|^3}{N}R.
\]
Substitution into \eqref{eq:popular-operator} gives
\[
 E_{4/3}^+(B)^3
 \lesssim
 \frac{S^8M^6|E|^3\mu^4\nu^4R^3}{N^{21}}.
\]
Since
\[
 |E|\mu
 \le
 \sum_{e\in E}r_{A-F}(e)
 \le N|F|
\]
and $\mu\le N$, one has
$|E|^3\mu^4=(|E|\mu)^3\mu\le N^3|F|^3\mu$.  On the other hand,
$E_{4/3}^+(B)^3\gtrsim |F|^3\nu^4$.  Cancelling
$|F|^3\nu^4$ gives
\[
 N^{18}\lesssim S^8M^6\mu R^3.
\]
The bound $\mu\le N$ now proves \eqref{eq:main-product}.  Since
$S^8M^6\le\max\{S,M\}^{14}$, the first inequality in
\eqref{eq:main-max} follows.  The second follows from
\[
 (1+N^3/p^2)^{-3/14}
 \gg
 \min\{1,(p^2/N^3)^{3/14}\}.
\]
\end{proof}

Corollary~\ref{cor:seventeen-fourteen} follows from
$N^3/p^2\le1$.  Corollary~\ref{cor:small-doubling} follows by inserting
$S\le KN$ into \eqref{eq:main-product} and taking sixth roots.


\begin{thebibliography}{99}

\bibitem{AralaChow}
N.~Arala and S.~Chow,
\emph{Incidence geometry and polynomial expansion over finite fields},
Proc. Roy. Soc. Edinburgh Sect. A (2026), 1--12,
doi:10.1017/prm.2026.10130.

\bibitem{BKT}
J.~Bourgain, N.~Katz and T.~Tao,
\emph{A sum-product estimate in finite fields, and applications},
Geom. Funct. Anal. \textbf{14} (2004), no.~1, 27--57.

\bibitem{BukhTsimerman}
B.~Bukh and J.~Tsimerman,
\emph{Sum-product estimates for rational functions},
Proc. Lond. Math. Soc. (3) \textbf{104} (2012), no.~1, 1--26,
doi:10.1112/plms/pdr018.

\bibitem{ErdosSzemeredi}
P.~Erd\H{o}s and E.~Szemer\'edi,
\emph{On sums and products of integers},
in \emph{Studies in Pure Mathematics}, Birkh\"auser, Basel, 1983,
213--218.

\bibitem{Garaev}
M.~Z.~Garaev,
\emph{An explicit sum-product estimate in $\mathbb F_p$},
Int. Math. Res. Not. IMRN (2007), Art.~ID rnm035.

\bibitem{HartLiShen}
D.~Hart, L.~Li and C.-Y.~Shen,
\emph{Fourier analysis and expanding phenomena in finite fields},
Proc. Amer. Math. Soc. \textbf{141} (2013), no.~2, 461--473,
doi:10.1090/S0002-9939-2012-11338-3.

\bibitem{HegyvaryHennecart}
N.~Hegyv\'ari and F.~Hennecart,
\emph{Conditional expanding bounds for two-variable functions over prime fields},
European J. Combin. \textbf{34} (2013), no.~8, 1365--1382,
doi:10.1016/j.ejc.2013.05.021.

\bibitem{KMPV}
D.~Koh, H.~Nassajian Mojarrad, T.~Pham and C.~Valculescu,
\emph{Four-variable expanders over the prime fields},
Proc. Amer. Math. Soc. \textbf{146} (2018), no.~12, 5025--5034,
doi:10.1090/proc/14177.

\bibitem{Mirzaei}
M.~Mirzaei,
\emph{A note on conditional expanders over prime fields},
Discrete Math. \textbf{343} (2020), no.~9, 111951,
doi:10.1016/j.disc.2020.111951.

\bibitem{MohammadiStevensLow}
A.~Mohammadi and S.~Stevens,
\emph{Low-energy decomposition results over finite fields},
arXiv:2102.01655v3 (2021).

\bibitem{MohammadiStevens}
A.~Mohammadi and S.~Stevens,
\emph{Attaining the exponent $5/4$ for the sum-product problem in finite fields},
Int. Math. Res. Not. IMRN \textbf{2023} (2023), no.~4, 3516--3532,
doi:10.1093/imrn/rnab338.

\bibitem{MojarradPham}
H.~Nassajian Mojarrad and T.~Pham,
\emph{Conditional expanding bounds for two-variable functions over arbitrary fields},
J. Number Theory \textbf{186} (2018), 137--146,
doi:10.1016/j.jnt.2017.09.020.

\bibitem{PhamVinhDeZeeuw}
T.~Pham, L.~A.~Vinh and F.~de Zeeuw,
\emph{Three-variable expanding polynomials and higher-dimensional distinct distances},
Combinatorica \textbf{39} (2019), no.~2, 411--426,
doi:10.1007/s00493-017-3773-y.

\bibitem{RRSh}
O.~Roche-Newton, M.~Rudnev and I.~D.~Shkredov,
\emph{New sum-product type estimates over finite fields},
Adv. Math. \textbf{293} (2016), 589--605,
doi:10.1016/j.aim.2016.02.019.

\bibitem{Rudnev}
M.~Rudnev,
\emph{On the number of incidences between points and planes in three dimensions},
Combinatorica \textbf{38} (2018), no.~1, 219--254,
doi:10.1007/s00493-016-3329-6.

\bibitem{RudnevShakanShkredov}
M.~Rudnev, G.~Shakan and I.~D.~Shkredov,
\emph{Stronger sum-product inequalities for small sets},
Proc. Amer. Math. Soc. \textbf{148} (2020), no.~4, 1467--1479,
doi:10.1090/proc/14902.

\bibitem{ShakanShkredov}
G.~Shakan and I.~D.~Shkredov,
\emph{Breaking the $6/5$ threshold for sums and products modulo a prime},
arXiv:1806.07091 (2018).

\bibitem{StevensDeZeeuw}
S.~Stevens and F.~de Zeeuw,
\emph{An improved point-line incidence bound over arbitrary fields},
Bull. Lond. Math. Soc. \textbf{49} (2017), no.~5, 842--858,
doi:10.1112/blms.12077.

\bibitem{Tao}
T.~Tao,
\emph{Expanding polynomials over finite fields of large characteristic,
and a regularity lemma for definable sets},
Contrib. Discrete Math. \textbf{10} (2015), no.~1, 22--98.

\bibitem{TranNguyen}
P.~D.~Tran and N.~Van The,
\emph{Some sum-product type estimates for two-variables over prime fields},
Discrete Appl. Math. \textbf{289} (2021), 1--11,
doi:10.1016/j.dam.2020.09.017.

\bibitem{Vinh}
L.~A.~Vinh,
\emph{The Szemer\'edi--Trotter type theorem and the sum-product estimate in finite fields},
European J. Combin. \textbf{32} (2011), no.~8, 1177--1181,
doi:10.1016/j.ejc.2011.06.008.

\bibitem{Vu}
V.~H.~Vu,
\emph{Sum-product estimates via directed expanders},
Math. Res. Lett. \textbf{15} (2008), no.~2, 375--388.

\end{thebibliography}
\end{document}